\documentclass{amsart}
\usepackage{amsmath,amssymb,amsthm,color,url}
\newtheorem{thr}{Theorem}
\newtheorem{lem}[thr]{Lemma}
\newtheorem{cor}[thr]{Corollary}
\newtheorem{pr}[thr]{Proposition}

\theoremstyle{definition}

\newtheorem{ex}[thr]{Example}

\theoremstyle{remark}

\def\F{\mathcal{F}}
\def\H{\mathcal{H}}
\def\M{\mathcal{M}}
\def\T{\mathcal{T}}
\def\Q{\mathbb{Q}}
\def\R{\mathbb{R}}
\def\Rc{\mathcal{R}}
\def\fa{\operatorname{fact}_+}
\def\rp{\operatorname{rk}_+}

\begin{document}

\title[A universality theorem for NMF]{A universality theorem for \\ nonnegative matrix factorizations}

\author{Yaroslav Shitov}
\email{yaroslav-shitov@yandex.ru}


\begin{abstract}
Let $A$ be a matrix with nonnegative real entries. A nonnegative factorization of size $k$ is a representation of $A$ as a sum of $k$ nonnegative rank-one matrices. The space of all such factorizations is a bounded semialgebraic set, and we prove that spaces arising in this way are universal. More presicely, we show that every bounded semialgebraic set $U$ is rationally equivalent to the set of nonnegative size-$k$ factorizations of some matrix $A$ up to a permutation of matrices in the factorization. We prove that, if $U\subset\mathbb{R}^n$ is given as the zero locus of a polynomial with coefficients in $\mathbb{Q}$, then such a pair $(A,k)$ can be computed in polynomial time. This result gives a complete description of the algorithmic complexity of nonnegative rank, and it also allows one to solve the problem of Cohen and Rothblum on nonnegative factorizations restricted to matrices over different subfields of $\R$.
\end{abstract}

\maketitle

\section{Introduction}

Let $A$ be a matrix with nonnegative real entries, and let $k$ be an integer number. The \textit{nonnegative matrix factorization}, or \textit{NMF}, is the task to approximate (or express) $A$ with a sum of $k$ rank-one matrices each of which has nonnegative entries. Both the general version of NMF and its exact counterpart are important tools in modern pure and applied mathematics. This problem is inherent in clustering problems and data mining~\cite{DHS}, and its outputs can be easier to interpret than those obtained from other factorization techniques~\cite{Gillis}. This feature leads to many real-world applications of the NMF, which include image processing~\cite{LS}, text mining~\cite{textmin}, music analysis~\cite{musan}, and audio signal processing~\cite{GNHS}. Another notable application of NMF is a parts-based approach to representation of objects, which leads to progress in face recognition~\cite{LS}. Also, the NMF problem does naturally arise in statistics and helps in studying the expectation maximization algorithms~\cite{KRS}, describes the spaces of explanations~\cite{MSS}, and one of its approximate versions turns out to be equivalent to a popular document clustering method known as the probabilistic latent semantic analysis~\cite{DLP}. A geometric approach allows one to reformulate NMF as a question on the extension complexity of a polytope and as a nested polytope problem, which leads to applications in combinatorial optimization~\cite{Yan}, theory of computation~\cite{FMPTdW}, statistics~\cite{KRS}, quantum mechanics~\cite{CR} and other fields.

\smallskip

As we see, the need of solving a particular instance of NMF can arise in a variety of applications, and this fact motivates the following problem.

\begin{pr}\label{prmain}
What is the computational complexity of NMF?
\end{pr}

This problem survived an extensive discussion during the last three decades. The NP-hardness of NMF is relatively easy to prove, and such a result has been known (at least) since 1993, see Lemma 3.3 in~\cite{JR}. However, the NMF techniques have not yet gained much of their popularity by the time the paper~\cite{JR} appeared, and this paper seems to remain unnoticed by many authors working on the NMF problem (see also a discussion in~\cite{myCR}). A standard reference to the NP-hardness of NMF is an ingenious construction of Vavasis~\cite{Vavas}, who actually proves a stronger result that NMF remains NP-hard even when the factorization rank equals the conventional rank of a given matrix. However, best known theoretical algorithms for NMF are still based on quantifier elimination techniques~\cite{CR}, and the true complexity of NMF remained unknown. In particular, several widely known references mention the result of Vavasis as an \textit{NP-completeness} proof~\cite{Moitra2, BRRT, ConBook}---in contrary to the fact that his paper explicitly states this as an open problem. The paper~\cite{Moitra2} makes further progress on Problem~\ref{prmain} and gives an NMF algorithm that halts in polynomial time for every fixed factorization rank $k$.

\smallskip

In this paper, we prove the \textit{universality theorem} for NMF, which states that the space of nonnegative factorizations of a given nonnegative matrix can be arbitrarily complicated (see Section~2 for a precise statement). In other words, our result is the NMF analogue of the famous theorem of Mn\"{e}v~\cite{Mnev} on oriented matroids and many universality results on polytopes~\cite{RGZ}, Nash equilibria~\cite{Datta}, positive semidefinite factorizations~\cite{mypsd}, tensor rank decompositions~\cite{mytens, SS}, and other objects in algebra and geometry~\cite{AAM, Mato}. As a corollary of our result, one deduces answers to several widely known problems on nonnegative matrices; in partucular, we get a complete solution of Problem~\ref{prmain}. The \textit{Existential theory of the reals (ETR)} is the following task: Given a given system of equations and inequalities with rational coefficients, does it have a real solution? Since NMF is a particular instance of ETR, it can in the worst case only be as hard as ETR; we show that this is indeed the case.

\begin{thr}\label{corof1}
NMF is polynomial-time equivalent to ETR.
\end{thr}

In other words, NMF belongs to the class of problems that are known in the computational complexity theory as $\exists\mathbb{R}$-complete, see~\cite{AAM, Mato, mypsd}. In the following section, we will formulate our universality theorem and deduce from it Theorem~\ref{corof1} and several other important results, including the solution of the Cohen--Rothblum problem on the behavior of NMF with respect to different ordered fields. As a final remark of this introduction, we note that the situation with the complexity of NMF is similar to the one we have for so-called \textit{positive semidefinite} matrix factorizations---these problems are $\exists\mathbb{R}$-complete in general~\cite{mypsd}, become polynomial time solvable when restricted to bounded factorization rank~\cite{Moitra2, mysmallpsd}, and have unknown complexity for matrices whose conventional rank is bounded~\cite{GG}.

\section{The Univerality theorem and its consequences}

Let us switch to a more general setting and consider arbitrary ordered field $\F$ and real closed field $\Rc$ containing $\F$. The symbols $\F,\Rc$ have this meaning throughout the paper, and we note a particularly important case of $(\F,\Rc)=(\Q,\R)$. We denote by $\F_+$ the set of all nonnegative elements of $\F$, and we call a matrix $A$ over $\F$ \text{nonnegative} if its entries belong to $\F_+$. A nonnegative $k$-factorization of such a matrix is a family $(A_1,\ldots,A_k)$ of rank-one matrices with entries in $\Rc_+$ such that $A_1+\ldots+A_k=A$; any permutation of this family obviously remains a valid factorization. However, we do not want to think of such factorizations as different, and we define $\fa(A,k)$ as the set of all nonnegative $k$-factorizations of $A$ satisfying an additional assumption that $A_1\succ\ldots\succ A_k$, where $\succ$ denotes the lexicographic ordering on the set of matrices. The set $\fa(A,k)$ is a bounded semi-algebraic subset of $\Rc^{mnk}$, where $m$ and $n$ are the dimensions of the matrix $A$. That is, this set can be defined by a quantifier-free formula over $\F$, and the entries of matrices in the factorizations are bounded (for instance, by the maximal entry of $A$). The \textit{nonnegative rank} of $A$ is the smallest $k$ such that $\fa(A,k)$ is non-empty; we denote this quantity by $\rp(A)$.

Now we need to introduce a variation of the concept of rational equivalence of sets (see e.g.~\cite{Mnev}), which we need to formulate our Universality theorem. Let $U\subset\Rc^n$ and $V\subset\Rc^m$ be a pair of sets satisfying $\pi(V)=U$, where $\pi:\Rc^m\to\Rc^n$ is the natural projection onto the first $n(\leqslant m)$ coordinates. If there are functions $\varphi_1,\ldots,\varphi_m$ rational over $\F$ such that, for all $u=(u_1,\ldots,u_n)\in U$, the preimage $\pi^{-1}(u)$ is unique and equal to
$$\left(u_1,\ldots,u_n,\varphi_{n+1}(u),\ldots,\varphi_m(u)\right),$$
then such a $\pi$ is called a \textit{rational projection} between $U$ and $V$. 
A \textit{permutation mapping} from $\Rc^n$ to $\Rc^n$ is a mapping that is defined, for any permutation $\sigma$ on $n$ elements, as $(a_1,\ldots,a_n)\to\left(a_{\sigma(1)},\ldots,a_{\sigma(n)}\right)$.
We say that arbitrary sets $U_1\subset\Rc^n$ and $U_2\subset\Rc^m$ are \textit{strongly equivalent} if $(U_1,U_2)$ belongs to the equivalence relation generated by rational projections and permutation mappings. As we can see, any pair of strongly equivalent sets are rationally equivalent as well. Now we are ready to formulate the main result of our paper.


\medskip

\textbf{The Universality Theorem for Nonnegative Factorizations.}

Let $F$ be the zero locus of a polynomial $f\in\F[x_1,\ldots,x_n]$.

There are a matrix $M$ over $\mathcal{F}_+$ and an integer $k$ such that $\fa(M,k)$ is strongly equivalent to $F\cap[0,1]^n$. If $\mathcal{F}=\mathbb{Q}$, then one can find $M$ and $k$ in polynomial time\footnote{When we discuss the computational complexity, we work under quite a strong assumption that $f$ is encoded as a list of monomials with coefficients and exponents written in the unary.}.

\medskip

We proceed with applications of this theorem; one of these is a solution of a well known problem of Cohen and Rothblum~\cite{CR}. For $A$ a nonnegative matrix over $\F$, we define $\rp(A,\F)$ as the smallest $k$ such that $\fa(A,k)$ contains a point with coordinates in $\F$; this quantity is called the nonnegative rank of $A$ \textit{with respect to} $\F$. Cohen and Rothblum asked, how sensitive is the nonnegative rank to the field with respect to which it is computed? The universality theorem can be thought of as a complete answer to this question; the following corollary generalizes previously known partial solutions in~\cite{anotherproof, mydep, myCR}. (We note in passing that~\cite{mydep} gives some progress on the case of factorization rank equal to conventional rank, for which the universality result is yet to be discovered.)

\begin{cor}\label{corcohrot}
Let $\F$ be an ordered field and $\Rc$ its real closure. If $\F_1,\F_2$ are different fields satisfying $\F\subset\F_1,\F_2\subset\Rc$, then there is a matrix $A$ over $F$ such that $\operatorname{Rank}_+(A,\F_1)\neq\operatorname{Rank}_+(A,\F_2)$.
\end{cor}

\begin{proof}
Let $\alpha$ be an element that lies in exactly one of the fields $\F_1,\F_2$. Linear substitutions allow us to assume that $\alpha$ belongs to the interval $(0,1)$, and also that this interval contains no other root of the $\F$-minimal polynomial $f$ of $\alpha$. Applying the Universality theorem with $n=1$ to this polynomial, we get a desired matrix.
\end{proof}

What is the complexity of computing the nonnegative rank with respect to different fields? This question was asked in~\cite{BR, CR} in the case of rational nonnegative rank. The Universality theorem shows that this problem is polynomial-time equivalent to the following question: Does a given bounded semialgebraic set contain a rational point? A well known conjecture posits that the solubility of rational \textit{Diophantine equations} is undecidable~\cite{KR, Poon}, so the current state of knowledge does not allow one to compute $\operatorname{Rank}_+(M,\mathbb{Q})$ algorithmically. Of course, there is nothing specific about the field $\mathbb{Q}$ in this consideration, and the above mentioned conclusion as to complexity of nonnegative rank is valid for any ordered field. In the case of the reals, an even more precise description is possible: Theorem~\ref{corof1} states that computing the real nonnegative rank is as hard as to determine if any, not necessarily bounded, semi-algebraic set is non-empty. Theorem~\ref{corof1} follows from the Universality theorem above and Theorem~3.3 in~\cite{mypsd}, which states that the ETR problem remains $\exists\mathbb{R}$-complete even if restricted to a single polynomial equation not having any solution outside the unit cube.

\smallskip

The remaining part of the paper (Sections~3--4) is the proof of the Universality theorem. In Section~\ref{seclocvar}, we introduce a gadget that allows us to replace some entries in matrices with unknowns, which leads to a problem richer than the conventional NMF. In Section~\ref{secfinproof}, we explain how to encode a polynomial in terms of the new problem and finalize our proof.

\section{The proof: A `variable gadget'}\label{seclocvar}

In this section, we explain the connection between the classical NMF problem and its generalization in which matrices are allowed to contain unknown entries. In other words, we consider so called \textit{incomplete matrices} which may contain not only elements of $\F_+$ but also variables whose ranges are segments in $\F_+$. Let $\mathcal{A}$ be such a matrix; a completion of $A$ is a matrix that can be obtained from $\mathcal{A}$ by assigning some value to every variable within its range. We define $\fa(\mathcal{A},k)$ as the union of all sets $\fa(A,k)$ over all completions $A$ of $\mathcal{A}$. The smallest $k$ for which $\fa(\mathcal{A},k)$ is non-empty is called the \textit{nonnegative rank} of an incomplete matrix $\mathcal{A}$. Let us illustrate these definitions with an example.

\begin{ex}\label{exex}
Consider the incomplete matrix $$\mathcal{A}=
\begin{pmatrix}
1&1&x_1\\
x_2&y&1\\
x_3&2&y
\end{pmatrix},$$
assuming that the variables range in $[0,2]$. This matrix is rank-one if $x_1=\sqrt{0.5}$, $x_2=y=\sqrt{2}$, $x_3=2$, so $\mathcal{A}$ has nonnegative rank one as an incomplete matrix. We also see that $\fa(\mathcal{A},1)$ is a singleton set corresponding to $(x_1,x_2,x_3,y)=\left(\sqrt{0.5},\sqrt{2},2,\sqrt{2}\right)$.
\end{ex}

This example already shows that the space of nonnegative factorization may be non-empty but consist of non-rational points only, even for incomplete matrices with rational known entries. Moreover, our first counterexample to the Cohen--Rothblum problem (see~\cite{myCR}) was essentially obtained by applying the gadget described below to the incomplete matrix as in Example~\ref{exex}. Let
$$A=\left(\begin{array}{ccc|ccc}
&&&&&\\
&\overline{A}&&&\textcolor{black}{V}&\\
&&&&&\\\hline
&\textcolor{black}{c}&&\textcolor{blue}{N}&\textcolor{blue}{\ldots}&\textcolor{blue}{N}\\
\end{array}\right)$$
be a nonnegative matrix. We say that we obtain the matrix 
$$G=\left(\begin{array}{ccc|ccc|cccc}
&&&&&&0&0&0&0\\
&\overline{A}&&&\textcolor{black}{V}&&\vdots&\vdots&\vdots&\vdots\\
&&&&&&0&0&0&0\\\hline
&\textcolor{black}{c}&&\textcolor{blue}{N}&\textcolor{blue}{\ldots}&\textcolor{blue}{N}&\textcolor{red}{M}&\textcolor{red}{M}&\textcolor{red}{M}&\textcolor{red}{M}\\\hline
0&\ldots&0&\textcolor{red}{M}&\textcolor{red}{\ldots}&\textcolor{red}{M}&\textcolor{red}{M}&\textcolor{red}{M}&\textcolor{red}{0}&\textcolor{red}{0}\\
0&\ldots&0&\textcolor{red}{0}&\textcolor{red}{\ldots}&\textcolor{red}{0}&\textcolor{red}{0}&\textcolor{red}{M}&\textcolor{red}{M}&\textcolor{red}{0}\\
0&\ldots&0&\textcolor{red}{0}&\textcolor{red}{\ldots}&\textcolor{red}{0}&\textcolor{red}{0}&\textcolor{red}{0}&\textcolor{red}{M}&\textcolor{red}{M}\\
0&\ldots&0&\textcolor{red}{0}&\textcolor{red}{\ldots}&\textcolor{red}{0}&\textcolor{red}{M}&\textcolor{red}{0}&\textcolor{red}{0}&\textcolor{red}{M}
\end{array}\right)
$$
by applying the \textit{variable gadget with parameter} $M$ to the blue entries of $A$. Let $\mathcal{A}$ be the matrix obtained from $A$ by replacing the blue entries with a variable $x$ ranging in $[\max\{0,N-M\},N]$. The following is essentially Theorem~2 in~\cite{myCR}.

\begin{thr}\label{pr3}
For $M,N\in\F_+$ we have

\noindent (1) if $r$ is an integer such that $r\leqslant \rp\left(\overline{A}\right)$, then $\fa(G,r+4)$ and $\fa(\mathcal{A},r)$ are strongly equivalent;

\noindent (2) if $\textcolor{black}{V}$ is a single column, then $\operatorname{Rank}_+(G,\F)=\operatorname{Rank}_+(\mathcal{A},\F)+4$. 
\end{thr}



The above theorem allows us to encode incomplete matrices in which the use of every variable is limited to a single row or column; now we are going to construct a more general gadget. Let $B$ be an incomplete matrix whose entries are elements of $\F_+$ and variables $x,x_1,\ldots,x_n$. We denote the set of entries equal to $x$ by $\T=\{(i_1,j_1),\ldots,(i_\tau,j_\tau)\}$, and we assume that the sequences $i_1,\ldots,i_\tau$ and $j_1,\ldots,j_\tau$ do not contain repeating indexes. Assume $M,N,P,Q\in\F$ are such that $N\geqslant Q>P>0$ and $M\geqslant1/P$. We construct the matrix 
$$\left(\begin{array}{cccc|ccc|ccc|ccc}
\textcolor{red}{N}&\textcolor{red}{N}&\textcolor{red}{0}&\textcolor{red}{0}&&&&&&&&&\\
\textcolor{red}{0}&\textcolor{red}{N}&\textcolor{red}{N}&\textcolor{red}{0}&&O&&&O&&&O&\\
\textcolor{red}{0}&\textcolor{red}{0}&\textcolor{red}{N}&\textcolor{red}{N}&&&&&&&&&\\\hline
\textcolor{red}{N}&\textcolor{red}{0}&\textcolor{red}{0}&\textcolor{red}{N}&\textcolor{red}{N}&\textcolor{red}{\ldots}&\textcolor{red}{N}&0&\ldots&0&0&\ldots&0\\\hline
\textcolor{red}{N}&\textcolor{red}{N}&\textcolor{red}{N}&\textcolor{red}{N}&\textcolor{green}{N}&\textcolor{green}{\ldots}&\textcolor{green}{N}&1&\ldots&1&0&\ldots&0\\\hline
&&&&&&&&&&&&\\
&O&&&&\textcolor{blue}{QI_\tau}&&&I_\tau&&&O&\\
&&&&&&&&&&&&\\\hline
&&&&&&&&&&&&\\
&O&&&&I_\tau&&&\textcolor{magenta}{B_\T}&&&B_2&\\
&&&&&&&&&&&&\\\hline
&&&&&&&&&&&&\\
&O&&&&O&&&B_3&&&B_4&\\
&&&&&&&&&&&&
\end{array}\right)
$$
and denote it by $\Gamma_1(B,x,M,N,P,Q)$. Here, the rows and columns of $\textcolor{magenta}{B_\T}$ have labels $i_1,\ldots,i_\tau$ and $j_1,\ldots,j_\tau$, respectively, and the matrix $\left(\begin{smallmatrix}\textcolor{magenta}{B_\T}&B_2\\B_3&B_4\end{smallmatrix}\right)$ is obtained from $B$ by replacing every entry equal to $x$ by $M$. (Therefore, the diagonal entries of $\textcolor{magenta}{B_\T}$ become equal to $M$, and the matrix $\Gamma_1$ does not anymore depend on $x$.) 
We denote by $\Gamma(B,x,M,N,P,Q)$ the matrix obtained from $\Gamma_1$ by applying the variable gadgets with parameter $Q-P$ to every diagonal entry of the blue block $\textcolor{blue}{QI_\tau}$.

\begin{thr}\label{lem1}
Let $\rp(B_4)\geqslant\rho$ and $\rho'=\rho+5\tau+4$. Then $\fa(\Gamma,\rho')$ and $\fa(B,\rho)$ are strongly equivalent, provided that $x\in[M-1/P,M-1/Q]$.
\end{thr}

\begin{proof}
Let $\Gamma_2$ be the matrix obtained from $\Gamma_1$ as follows. We remove the first four rows and the first four colums, we replace the green entries by the variable $\textcolor{green}{c\in[0,N]}$, we replace the diagonal entries of $\textcolor{blue}{QI_\tau}$ by different variables $\textcolor{blue}{v_i\in[P,Q]}$. It can be seen from Theorem~\ref{pr3} that $\fa(\Gamma,\rho')$ is strongly equivalent to $\fa(\Gamma_2,\rho+\tau)$.

Now let $(G_1,\ldots,G_{\rho+\tau})$ be a nonnegative factorization of $\Gamma_2$. We easily get $\textcolor{blue}{v_i}=\textcolor{green}{c}$, and if $G_j$ is a matrix with a non-zero blue entry, then it has to be equal to $\textcolor{green}{c}$, and we have $1/\textcolor{green}{c}$ at the corresponding entry of $G_j$ located on the diagonal of the \textcolor{magenta}{magenta} block. So we can see that $\fa(\Gamma_2,\rho+\tau)$ is strongly equivalent to $\fa(B,\rho)$, since $1/\textcolor{green}{c}=1/\textcolor{blue}{v_i}$ ranges in $[1/P,1/Q]$.
\end{proof}

The main result of this section is a natural generalization of Proposition~\ref{lem1}.

\begin{thr}\label{thr1}
Let $\H$ be an incomplete matrix whose entries are elements of $\F_+$ and variables $x_1,\ldots,x_n$. Assume that $x_i$ ranges in $[a_i,b_i]\subset\F_+$ and occurs at most once in every row and column. 
Let $H_0$ be the matrix obtained from $\H$ by removing all the rows and columns in which the variables occur, and suppose $\rp(H_0)=\rho$.

Then there is a complete matrix $\M=\M(\H)$ for which the sets $\fa(\M,\rho+4n+5\tau_1+\ldots+5\tau_n)$ and $\fa(\H,\rho)$ are strongly equivalent, where $\tau_i$ is the number of times that $x_i$ occurs. If $\F=\mathbb{Q}$, then we can find such an $\M$ in time polynomial in the size of the input.
\end{thr}

\begin{proof}
Repeatedly apply Theorem~\ref{lem1} to eliminate variables $x_1,\ldots,x_n$.
\end{proof}

\section{The proof: Encoding a polynomial equation}\label{secfinproof}

Let us complete the proof of the Universality theorem using Theorem~\ref{thr1}. The following lemmas explain how to express the polynomial equation $f=0$ in terms of the factorization spaces of incomplete matrices. Namely, one of them encodes the equations involving a linear combination, and the other allows one to express the product of two variables.

\begin{lem}\label{lemsum}
Let $y_1,\ldots,y_l$ be variables ranging in $[0,1]$, let $s_1,\ldots,s_l,N\in\F_+$ satisfy $N\geqslant s_1+\ldots+s_l$, and let $L$ be a variable ranging in $[0,N]$. Let $S=S(s_1y_1+\ldots+s_ly_l=L)$ be the matrix obtained from
$$\left(\begin{array}{c|ccc}
y_1&&&\\
\vdots&&I_l&\\
y_l&&&\\\hline
L&s_1&\ldots&s_l\\\hline
\textcolor{blue}{1}&&&\\
\textcolor{blue}{\vdots}&&I_l&\\
\textcolor{blue}{1}&&&
\end{array}\right)
$$
by applying the variable gadgets with parameter $1$ to every single blue entry. Then

\noindent (1) if we remove all the rows and columns of $S$ which contain variables, we get a matrix with nonnegative rank $5l$;

\noindent (2) for any fixed assignment of values to variables $y_1,\ldots,y_l$, there is a unique nonnegative $5l$-factorization of $S$ corresponding to the value of $L$ equal to $s_1y_1+\ldots+s_ly_l$. Some entries of matrices in this factorization are equal to $y_1,\ldots,y_l$, and the rest are rational functions of these variables fixed in advance. 
\end{lem}

\begin{proof}
We use Theorem~\ref{pr3}(2) and replace the blue entries with variables; this will decrease the nonegative rank by $4l$ and save the nonnegative $5l$-factorization space up to a rational projection. Now, the condition (1) is valid because of the bottom blocks of the matrix, and (2) can be checked straightforwardly.
\end{proof}

\begin{lem}\label{lemprod}
Let $v,u_1,u_2$ be variables ranging in $[0,1]$. Let $P=P(u_1u_2=v)$ be the matrix obtained from
$$\left(\begin{array}{ccc}
v&u_1&\textcolor{blue}{1}\\
u_2&1&1\\
\textcolor{blue}{1}&1&1
\end{array}\right)
$$
by applying the variable gadgets with parameter $1$ to the blue entries. If we remove all the rows and columns of $S$ which contain variables, we get a matrix with nonnegative rank $9$. For any fixed assignment of values to variables $u_1,u_2$, there is a unique nonnegative $9$-factorization of $S$ corresponding to the value of $v$ equal to $u_1u_2$. Some entries of matrices in this factorization are equal to $u_1,u_2$, and the rest are rational functions of these variables fixed in advance. 
\end{lem}

\begin{proof}
Similarly to the previous lemma.
\end{proof}

Now we are ready to prove the Universality theorem. For any polynomial $f\in\F[x_1,\ldots,x_n]$, we can write the equation $f=0$ as $s_1\mu_1+\ldots+s_l\mu_l=s_{l+1}\mu_{l+1}+\ldots+s_r\mu_r$, where $s_i\in\F_+$ and $\mu_i$ are products of the form $\alpha_{i1}\ldots \alpha_{im_i}$ with $\alpha_{ij}\in\{x_1,\ldots,x_n\}$. We introduce new variables $L$ and $v_{ij}$, where $i\in\{1,\ldots,r\}$ and $j\in\{2,\ldots,i_m\}$; we also set $v_{i1}=\alpha_{i1}$.

Now we are going to construct an incomplete matrix $\H$ depending on the $x_i$'s, $v_{ij}$'s and $L$. As in the Lemmas~\ref{lemsum} and~\ref{lemprod}, we assume that $x_i$'s and $v_{ij}$'s are ranging in $[0,1]$, and also $L$ is ranging in $[0,s_1+\ldots+s_r]$. We define $\H$ as the block-diagonal matrix containing the following matrices as diagonal blocks: $S(L=s_1v_{1m_1}+\ldots+s_lv_{lm_l})$, $S(L=s_{l+1}v_{l+1,m_{l+1}}+\ldots+s_rv_{rm_r})$, and all the matrices $P(v_{ij}=v_{i,j-1}\alpha_{ij})$ for all $i$ and all $j\in\{2,\ldots,m_i\}$.

If we remove from $\H$ all the rows and columns in which the variables occur, we get the matrix of nonnegative rank $\rho=5r+9(m_1+\ldots+m_r-r)$. By Lemmas~\ref{lemsum} and~\ref{lemprod}, the set $\fa(\H,\rho)$ is strongly equivalent to the intersection of the cube $[0,1]^n$ and the zero locus of $f$. It remains to apply Theorem~\ref{thr1} to complete the proof of the Universality theorem.

\end{document}